\documentclass[10pt,letterpaper]{article}

\usepackage[T1]{fontenc}
\usepackage{lmodern}
\usepackage{microtype}
\usepackage{amsmath,amssymb,amsthm,mathtools}
\usepackage{geometry}
\usepackage[hidelinks,hyperfootnotes=false]{hyperref}

\allowdisplaybreaks

\newtheorem{theo}{Theorem}[section]
\newtheorem{conj}[theo]{Conjecture}
\newtheorem{ques}[theo]{Question}
\newtheorem{Lemma}[theo]{Lemma}
\theoremstyle{remark}
\newtheorem{remk}[theo]{Remark}

\newcommand{\N}{\mathbb N}
\newcommand{\fl}[1]{\left\lfloor #1\right\rfloor}
\newcommand{\ce}[1]{\left\lceil #1\right\rceil}
\renewcommand{\a}{\fl{\frac{\Delta}{2}}}
\renewcommand{\b}{\ce{\frac{\Delta}{2}}}

\hypersetup{
  pdftitle={The Maximum Number of Shortest Paths in Graphs},
  pdfauthor={Jing Yu and Jie-Xiang Zhu},
  pdfkeywords={shortest paths, extremal graph theory, bounded-degree graphs, random walks}
}

\begin{document}

\begin{center}
{\Large\bfseries THE MAXIMUM NUMBER OF SHORTEST PATHS IN GRAPHS\par}
\vspace{1.2em}
{\large JING YU\textsuperscript{b} AND JIE-XIANG ZHU\textsuperscript{a,*}\par}
\vspace{0.75em}
{\small \textsuperscript{a}Department of Mathematics, Shanghai Normal University, Shanghai, China\par}
{\small \textsuperscript{b}Shanghai Center for Mathematical Sciences, Fudan University, Shanghai, China\par}
\end{center}

\begingroup
\renewcommand{\thefootnote}{}
\renewcommand{\footnoterule}{%
  \kern -3pt
  \hrule width 2cm
  \kern 1pt
}
\footnotetext{* Corresponding author.}
\footnotetext{E-mail addresses: \texttt{jyu@fudan.edu.cn} (J. Yu), \texttt{zhujx@shnu.edu.cn} (J.-X. Zhu).}
\endgroup

\begin{center}
\begin{minipage}{0.89\textwidth}
\small
\noindent\textbf{Abstract.}
Benjamini and Tzalik obtained an upper bound on the number of shortest paths between two vertices at distance $t$ in a multigraph of maximum degree at most $\Delta$, and proposed a conjecture on the sharp bound. In this paper, we develop a probabilistic counting argument based on probability distributions induced by random walks from the two endpoints. This approach yields a sharp bound for multigraphs and confirms their conjecture. We further determine the exact maximum for simple graphs and thus answer another question of Benjamini and Tzalik. We also investigate the equality cases, describing the structure of the subgraph formed by shortest paths between $x$ and $y$ and giving tight examples.
\end{minipage}
\end{center}

\medskip
\noindent\textbf{AMS Subject Classification (2020):} Primary 05C35; Secondary 05C12, 60C05.

\noindent\textbf{Keywords:} shortest paths, extremal graph theory, bounded-degree graphs, random walks.

\section{Introduction}
Throughout, we work with multigraphs. Degrees are counted with multiplicity, and paths are distinguished by their edge sequences. Self-loops may be ignored, as they do not affect any of our results.
Let $G$ be a multigraph of maximum degree at most $\Delta$, and let distinct vertices $x, y \in V(G)$ be at distance $t \ge 1$. Since the case $\Delta = 1$ is trivial, we assume throughout that $\Delta \ge 2$. Let $n_G(x,y)$ denote the number of shortest paths between $x$ and $y$ in $G$. We refer to such a shortest path as an $x$--$y$ geodesic. Benjamini and Tzalik \cite{ben2023}  obtained an upper bound on $n_G(x, y)$ using an entropy argument, closely related to the entropy-based approach of Babu and Radhakrishnan \cite{BabuRadhakrishnan2014}. Their result improved upon the na\"ive $\Delta (\Delta-1)^{t-1}$ bound as follows:

\begin{theo}[{\cite[Theorem~1]{ben2023}}]\label{thm:BT}
    Let $G$ be a multigraph of maximum degree at most $\Delta$ and let $t$ be the distance in $G$ between two vertices $x,y$. Then \[n_G(x,y)\le \Delta\left(\a\b\right)^{(t-1)/2}.\]
\end{theo}

Moreover, for $t \ge 2$, Benjamini and Tzalik \cite{ben2023} gave a tight example for Theorem~\ref{thm:BT}. Let $C_{2t,\Delta}$ be the multigraph obtained from
the cycle $C_{2t}$ by replacing its edges alternately with $\a$ and $\b$ multiple edges.
They observed that $C_{2t,\Delta}$ attains the bound in Theorem~\ref{thm:BT} when $\Delta$ is even, as well as when both $\Delta$
and $t$ are odd. For the remaining case, they proposed the following
conjecture:
\begin{conj}[{\cite[Conjecture~2]{ben2023}}]\label{conj:BT}
    In the case of odd $\Delta$ and even $t$,
    \[n_G(x,y)\le 2\left(\a\b\right)^{t/2}.\]
The bound is sharp, as witnessed by $C_{2t,\Delta}$.
\end{conj}
We confirm Conjecture~\ref{conj:BT}. That is, we show the following result:
\begin{theo}\label{thm:main}
Let $G$ be a multigraph of maximum degree at most $\Delta$ and let $t$ be the distance in $G$ between two vertices $x,y$. Then
\begin{align*}
n_G(x,y) \le
\begin{cases}
\Delta\left(\a\b\right)^{(t-1)/2}, & \text{if $t$ is odd,} \\[2mm]
2\left(\a\b\right)^{t/2}, & \text{if $t$ is even.}
\end{cases}
\end{align*}
The bound is sharp.
For $t = 1$, equality is attained when $x$ and $y$ are joined by $\Delta$ multiple edges; for $t \ge 2$, it is attained by $C_{2t,\Delta}$.
\end{theo}
We also describe the structure of the subgraph formed by all $x$--$y$ geodesics when equality holds. Benjamini and Tzalik also suggested the following question:
\begin{ques}[{\cite[Question~3]{ben2023}}]\label{ques:BT}
    What is the exact maximum number $n_G(x,y)$ among all simple graphs in terms of the bound $\Delta$ on the degree, and the distance $t$ between $x$ and $y$?
\end{ques}
For simple graphs, when $t \ge 3$, their entropy argument also yields the upper bound
\[n_G(x,y) \le \Delta(\Delta - 1)\left(\a\b\right)^{(t-3)/2}.\] We answer Question~\ref{ques:BT} by determining the exact maximum in the following theorem:
\begin{theo}\label{thm:2}
Let $G$ be a simple graph of maximum degree at most $\Delta$ and let $t$ be the distance in $G$ between two vertices $x,y$. Then
\begin{align*}
n_G(x,y) \le
\begin{cases}
1, & \text{if $t=1$,} \\[2mm]
\Delta, & \text{if $t = 2$,} \\[2mm]
\Delta(\Delta - 1)\left(\a\b\right)^{(t-3)/2}, & \text{if $t\ge 3$ is odd,} \\[2mm]
2(\Delta -1)\left(\a\b\right)^{(t-2)/2}, & \text{if $t\ge 4$ is even.}
\end{cases}
\end{align*}
These upper bounds are sharp.
\end{theo}
We also investigate the equality cases and give tight examples; see Section~\ref{sec:thm2}.

A related but complementary extremal parameter was recently introduced by Knor, Sedlar, \v{S}krekovski, and Zhang~\cite{KSSZ26}: they sum the number of geodesics over all pairs of vertices and study extremal graphs of fixed order. Here the order is unrestricted, while the maximum degree and the distance of one distinguished pair are fixed. From a probabilistic perspective, Benjamini, Hoppen, Ofek, Pra{\l}at and Wormald~\cite{BHOW11} determined the limiting distribution of the number of geodesics between two random vertices in a random regular graph.

Before turning to the proofs, we briefly describe our approach. Let $\mathcal P$ denote the set of all $x$--$y$ geodesics. To estimate $|\mathcal P|$, we use a probabilistic counting argument. In spirit, our argument is close to Lubell's probabilistic proof of the LYM inequality \cite{lubell1966}. For any probability distribution $\mu$ on $\mathcal P$, it follows from
$$\sum_{P \in \mathcal P} \mu(P) = 1$$
that
\begin{equation} \label{eq:pro-counting}
 |\mathcal P| \le \frac{1}{\min_{P \in \mathcal P} \mu(P)}.
\end{equation}
Equality holds if and only if $\mu$ is the uniform distribution on $\mathcal P$. Therefore, it suffices to construct a probability distribution $\mu$ on $\mathcal P$ and find a constant $c>0$ such that
\[
\mu(P) \ge c \quad \text{for every $P \in \mathcal P$.}
\]
To describe our probability distribution, let $H$ be the subgraph consisting of all vertices and edges that belong to at least one $x$--$y$ geodesic. The graph $H$ admits a natural layering according to the distance from $x$. We then average the distributions on $\mathcal P$ induced by the forward and backward random walks along the layers of $H$. The resulting bidirectional random-walk construction allows us to obtain a pointwise lower bound on the probability of each geodesic, in contrast to the entropy method, which controls the average information content of a random geodesic. In the present setting, this approach leads to a refinement of the bound obtained from the entropy argument in \cite{ben2023}.

\section{Proof of Theorem~\ref{thm:main}}
\subsection{The upper bound}
Most of our notation is standard or should be clear from the context. In what follows, we set
\begin{align*}
q := \a\b = \begin{cases}
k(k+1), & \text{if } \Delta=2k+1,\\
k^2, & \text{if } \Delta=2k.
\end{cases}
\end{align*}
We first prove the following lemma, which will provide the key lower bound in our probabilistic counting argument, together with the corresponding equality conditions.

\begin{Lemma}\label{lem:reciprocal-product}
Let $\Delta, r \in \N$, and let $(a_1,b_1),\ldots,(a_{r},b_{r})$ be pairs of positive integers such that for every $1\leq i\leq r$,
\[
        a_i+b_i\leq \Delta.
\]
Set
\[
        A:=\prod_{i=1}^{r}a_i,
        \qquad
        B:=\prod_{i=1}^{r}b_i,
\]
and
\[
q := \a\b.
\]
Then
\[
\frac1A+\frac1B \geq
\begin{cases}
2 q^{-r/2},
& \text{if $r$ is even},\\[2mm]
\Delta q^{-(r+1)/2},
& \text{if $r$ is odd}.
\end{cases}
\]
If $\Delta = 2k$, equality holds if and only if each $(a_i, b_i)$ is $(k, k)$. If $\Delta = 2k+1$, equality holds if and only if each $(a_i, b_i)$ is either $(k, k+1)$ or $(k+1, k)$, and the numbers of occurrences of these two pairs are equal when $r$ is even and differ by exactly $1$ when $r$ is odd.
\end{Lemma}

\begin{proof}
Observe that \[a_i b_i\leq \a \b = q\] for every $i$.
Suppose first that $\Delta = 2k$. Then $q = k^2$, and hence
\[
AB = \prod_{i=1}^{r} a_i b_i \le q^{r}.
\]
Consequently,
\[
\frac1A+\frac1B \geq \frac{2}{\sqrt{AB}} \ge \frac{2}{q^{r/2}}.
\]
Equality holds only if equality holds in both inequalities above, that is to say, each $(a_i, b_i)$ is $(k, k)$.

We now consider the case $\Delta = 2k+1$, so that $q = k(k+1)$. If $a_i b_i<q$ for some $i$, then, since
$a_i b_i$ is an integer,
\[
a_i b_i\leq q-1,
\]
and therefore
\[
AB \le (q-1) q^{r-1}.
\]
Hence,
\[
\frac1A+\frac1B
        \geq \frac{2}{\sqrt{AB}}
        \geq
        \frac{2}{q^{(r-1)/2}\sqrt{q-1}}.
\]
When $r$ is even,
\[
\frac{2}{q^{(r-1)/2}\sqrt{q-1}}
>
2q^{-r/2};
\]
whereas when $r$ is odd,
\[
\frac{2}{q^{(r-1)/2}\sqrt{q-1}}
>
(2k+1)q^{-(r+1)/2},
\]
where the latter inequality follows from
\[
4q^2-(2k+1)^2(q-1)=3q+1>0.
\]

It remains to consider the case $a_i b_i=q$ for every $i$. The condition $a_i+b_i\leq 2k+1$ then implies
\[
        \{a_i,b_i\}=\{k,k+1\}.
\]
Let $s$ be the number of indices for which
$(a_i,b_i)=(k,k+1)$. Then
\[
        A=k^s(k+1)^{r-s},
        \qquad
        B=(k+1)^s k^{r-s}.
\]
Set $s' = 2s- r$, which is the difference between the numbers of occurrences of $(k,k+1)$ and $(k+1, k)$. Then
\[
\frac1A+\frac1B
=
q^{- r/2}
\left[
\left(\frac{k+1}{k}\right)^{s'/2}
+
\left(\frac{k}{k+1}\right)^{s'/2}
\right].
\]
For fixed $k > 0$, the function
\[
u\mapsto
\left(\frac{k+1}{k}\right)^u
+
\left(\frac{k}{k+1}\right)^u
\]
is strictly increasing on $[0,\infty)$. Hence $\frac1A+\frac1B$ is minimized when $|s'|$ is as small as possible. When $r$ is even, $s'$ is even and may equal $0$. Therefore,
\[
\frac1A+\frac1B\geq 2q^{-r/2},
\]
with equality if and only if $s'=0$. When $r$ is odd, we have that $s'$ is a nonzero odd integer so that $|s'| \ge 1$. Therefore,
\begin{align*}
\frac1A+\frac1B
\ge
q^{- r/2}
\left(
\sqrt{\frac{k+1}{k}}
+
\sqrt{\frac{k}{k+1}}
\right) = \Delta q^{-(r+1)/2},
\end{align*}
with equality if and only if $s' = \pm 1$.
This completes the proof.
\end{proof}

We now turn to the multigraph $G$. The case $t=1$ is immediate, so in what follows we assume that $t\geq2$. Delete all vertices and edges that do not lie on an $x$--$y$ geodesic, and denote the resulting graph by $H$. For $0 \le i \le t$, define the $i$th geodesic layer by
\[
V_i:= \left\{ v \in V(H) : d_H(x, v) = i \right\}.
\]
In particular,
\[
V_0 = \{x\}, \qquad V_t = \{y\}.
\]
Since each vertex of $H$ lies on an $x$--$y$ geodesic, we have
\[
V(H) = \bigsqcup_{i=0}^{t} V_i,
\]
and every edge of $H$ joins two consecutive layers. Moreover, every vertex $v \in V_i$, $1 \le i \le t-1$, has an incident edge to each of $V_{i-1}$ and $V_{i+1}$. For each $v\in V_i$ with $1\leq i\leq t-1$, let
\[
        d^-(v):=|E(v,V_{i-1})|,
        \qquad
        d^+(v):=|E(v,V_{i+1})|,
\]
where $E(v,V_j)$ denotes the set of edges of $H$ joining $v$ to vertices in $V_j$, counted with multiplicity. Then
\[
  d^-(v), \, d^+(v) \ge 1 \qquad \text{and} \qquad   d^-(v)+d^+(v)\leq \mathrm{deg}_G(v) \le \Delta.
\]
We also set
\[
d^+(x) := |E(x,V_{1})|,\qquad d^-(y) := |E(y,V_{t-1})|,
\]
Clearly,
$$d^+(x),\, d^-(y) \le \Delta.$$

Let $\mathcal P$ be the set of all $x$--$y$ geodesics. We next define two probability distributions on $\mathcal P$ induced by random walks from $x$ to $y$ and from $y$ to $x$, respectively. First, define the probability distribution $\mu_{x \to y}$ by the forward random walk along the geodesic layers from $x$ to $y$, where at each vertex an edge leading to the next layer is chosen uniformly at random. Thus, for each
\[
        P=(x=v_0,e_1,v_1,\ldots,e_{t},v_{t}=y) \in \mathcal P,
\]
we have
\[
\mu_{x \to y}(P)
        =
        \frac{1}{
        d^+(x)\prod_{i=1}^{t-1}d^+(v_i)}.
\]
The analogous backward random walk along the geodesic layers from $y$ to $x$ defines another probability distribution $\mu_{y \to x}$ satisfying
\[
        \mu_{y \to x}(P)
        =
        \frac{1}{
        d^-(y)\prod_{i=1}^{t-1}d^-(v_i)}.
\]

Set
\[
\mu := \frac12 \left( \mu_{x \to y} +  \mu_{y \to x} \right).
\]
Then $\mu$ is also a probability distribution on $\mathcal P$. Now fix
\[
P=(x=v_0,e_1,v_1,\ldots,e_{t},v_{t}=y) \in \mathcal P,
\]
set
\[
        a_i:=d^+(v_i),
        \quad
        b_i:=d^-(v_i), \quad 1 \le i \le t-1,
\]
and
\[
        A(P):=\prod_{i=1}^{t-1}a_i,
        \qquad
        B(P):=\prod_{i=1}^{t-1}b_i.
\]
Since $d^+(x),d^-(y)\leq \Delta$, Lemma~\ref{lem:reciprocal-product} with $r = t-1$ gives
\begin{align*}
\mu(P) &= \frac12 \left( \frac{1}{d^+(x)A(P)} + \frac{1}{d^-(y)B(P)}\right) \\
&\ge \frac{1}{2\Delta} \left( \frac{1}{A(P)} + \frac{1}{B(P)}\right)\\
& \ge \begin{cases}
\Delta^{-1} q^{-(t-1)/2},
& \text{if $t$ is odd},\\[2mm]
2^{-1} q^{-t/2},
& \text{if $t$ is even}.
\end{cases}
\end{align*}
Combining this with \eqref{eq:pro-counting}, we obtain
\[
n_G(x, y) = |\mathcal P| \le \begin{cases}
\Delta q^{(t-1)/2},
& \text{if $t$ is odd},\\[2mm]
2 q^{t/2},
& \text{if $t$ is even}.
\end{cases}
\]
If equality holds in this bound, then every $P\in\mathcal P$ must attain equality in the pointwise lower bound for $\mu(P)$. Hence equality must hold at every step of the estimate for $\mu(P)$.

\subsection{Equality cases and tight examples}

Now we consider the structure of $H$ when equality holds. We may restrict our attention to $H$, since vertices and edges of $G \setminus H$ do not affect $n_G(x,y)$. We distinguish three cases. When $\Delta$ is even, the structure of $H$ can be described explicitly. For odd $\Delta$, the equality conditions involve both local degree constraints and a global balance condition along geodesics, so we restrict ourselves to recording these conditions rather than giving a full classification.

\medskip

\noindent\textbf{Case 1: $\Delta=2k$.} By the equality conditions in Lemma~\ref{lem:reciprocal-product}, every $v \in V_i$ with $1 \le i \le t-1$ satisfies
\[
d^-(v) = d^+(v) = k.
\]
Together with $d^+(x) = d^-(y) = 2k$, this implies that for every $1\le i \le t-1$,
\[
|V_i| = 2.
\]

Let $E(V_j, V_{j+1})$ denote the set of edges of $H$ joining $V_j$ and $V_{j+1}$. The edge sets $E(V_0,V_1)$ and $E(V_{t-1},V_t)$ are then uniquely determined: each consists of $k$ multiple edges from the corresponding endpoint to each of the two vertices in the adjacent layer. Hence it remains to describe $E(V_i,V_{i+1})$ for $1\le i\le t-2$. After labeling the two vertices in each of $V_i$ and $V_{i+1}$, the edge set $E(V_i,V_{i+1})$ can be represented by a $2\times2$ matrix $M_i$ with nonnegative integer entries, each recording the multiplicity of the corresponding pair of vertices. By the equality conditions in Lemma~\ref{lem:reciprocal-product}, all row and column sums of $M_i$ equal $k$. Hence
$$ M_i= \begin{pmatrix} \rho_i & k-\rho_i\\ k-\rho_i & \rho_i \end{pmatrix} $$
for some integer $0\le \rho_i\le k$. Conversely, every choice of integers $0\leq \rho_i\leq k$ gives an equality case.
\medskip

\noindent\textbf{Case 2: $\Delta=2k + 1$ with $k \ge 2$.} For
$v\in V_i$ with $1\leq i\leq t-1$, we say that $v$ is of type $A$ if
\[
(d^-(v),d^+(v))=(k,k+1),
\]
and of type $B$ if
\[
(d^-(v),d^+(v))=(k+1,k).
\]
By the equality conditions in Lemma~\ref{lem:reciprocal-product}, every vertex of $V_i$, $1\leq i\leq t-1$, is of one of these two types; let $p_i$ and $q_i$ denote the respective numbers of vertices of types $A$ and $B$. We first determine the size of each layer. Since $d^+(x)=2k+1$, counting $|E(V_0, V_1)|$ gives
\[
k p_1+(k+1)q_1=2k+1.
\]
As $k \ge 2$, this implies $(p_1, q_1) = (1, 1)$. Similarly, for $1\leq i\leq t-2$, counting $|E(V_i,V_{i+1})|$ gives
\[
(k+1)p_i+kq_i
=
kp_{i+1}+(k+1)q_{i+1}.
\]
Starting from $(p_1,q_1)=(1,1)$ and proceeding inductively, we obtain
\[
p_i=q_i=1 \qquad\text{for every } 1\leq i\leq t-1.
\]
Equivalently, for every $1\le i \le t-1$,
\[
|V_i| = 2,
\]
with exactly one vertex of each type in $V_i$. In particular, $E(V_0,V_1)$ and $E(V_{t-1},V_t)$ are uniquely determined. Therefore, it remains to describe $E(V_i,V_{i+1})$ for $1\le i\le t-2$. Order the two vertices in each $V_i$ by type as $(A,B)$. As in the previous case, $E(V_i,V_{i+1})$ is represented by a $2\times2$ matrix $M_i$ of the form
$$ M_i= \begin{pmatrix} \rho_i & k+1 - \rho_i\\ k-\rho_i & \rho_i \end{pmatrix} $$
for some integer $0\le \rho_i\le k$. Unlike the even-$\Delta$ case, where each $\rho_i$ can be chosen freely, the equality conditions in Lemma~\ref{lem:reciprocal-product} impose an additional global constraint in the odd-$\Delta$ case: along every $x$--$y$ geodesic, the numbers of vertices of types $A$ and $B$ lying in $V_1,\ldots,V_{t-1}$ must be equal when $t$ is odd and differ by exactly $1$ when $t$ is even.
\medskip

\noindent\textbf{Case 3: $\Delta=3$.}
We retain the notation of Case~2. The subtlety in this case is that
\[
p_1+2q_1=3
\]
has two nonnegative integer solutions, namely $(p_1, q_1) = (1, 1)$ or $(3, 0)$. Similarly, counting $|E(V_{t-1},V_t)|$ gives
\[
2p_{t-1}+q_{t-1}=3,
\]
and hence $(p_{t-1}, q_{t-1}) = (1, 1)$ or $(0, 3)$.
For $1\leq i\leq t-2$, the same counting argument as above gives
\[
2p_i+q_i=p_{i+1}+2q_{i+1}.
\]
Thus the sequence $(p_i,q_i)$ must satisfy these recurrence and boundary conditions. Once the sequence $(p_i,q_i)$, $1\le i\le t-1$, is fixed, the edge sets $E(V_0,V_1)$ and $E(V_{t-1},V_t)$ are uniquely determined. It remains to describe $E(V_i, V_{i+1})$ for $1 \le i \le t-2$. After ordering the vertices in $V_i$ and
$V_{i+1}$ by type, the edge set $E(V_i,V_{i+1})$ is represented by a
$(p_i+q_i)\times(p_{i+1}+q_{i+1})$ matrix with nonnegative integer entries. Its row sums are $2$ for vertices of type $A$ and $1$ for vertices of type $B$, while its column sums are $1$ for vertices of type $A$ and $2$ for vertices of type $B$. As in Case~2, the global constraint still holds: along every $x$--$y$ geodesic, the numbers of vertices of types $A$
and $B$ lying in $V_1,\ldots,V_{t-1}$ must be equal when $t$ is odd and differ by exactly $1$ when $t$ is even.

\medskip

Finally, as a simple special case, when $H=C_{2t,\Delta}$, one readily checks that the equality conditions in Lemma~\ref{lem:reciprocal-product} are satisfied and that $d^+(x) = d^-(y) = \Delta$. Hence $C_{2t, \Delta}$ is a tight example.

\begin{remk}
We conclude this section with a simple observation concerning the extremal example $C_{2t,\Delta}$, which may be of independent interest. Fix antipodal vertices $x,y$, and choose an $x$--$y$ geodesic according to the probability distribution $\mu$. There are only two possible vertex sequences, corresponding to the two directions around the underlying cycle. Conditional on either direction, the choices of the multiple edges at the $t$ successive steps are independent and uniform. After rescaling the edge labels to lie in $[0,1]$, the distribution $\mu$ induces, via this parametrization, a probability measure $\nu_\Delta$ on
\[ \{-1,+1\}\times[0,1]^t. \]
For fixed $t$, as $\Delta\to\infty$,
\[ \nu_\Delta\Rightarrow \nu, \]
where $\nu$ is the product of the uniform probability measure on $\{-1,+1\}$ and Lebesgue probability measure on $[0,1]^t$.
\end{remk}

\section{Proof of Theorem~\ref{thm:2}} \label{sec:thm2}

\subsection{The upper bound}
In this section we consider the case where $G$ is a simple graph. We retain the notation from the previous section. In particular, let $H$ be the subgraph obtained by deleting all vertices and edges that do not lie on an $x$--$y$ geodesic, and let $V_0, \ldots,V_t$ be the corresponding geodesic layers. The cases $t \le 2$ are immediate. We first consider $t=3$. In this case, every
$x$--$y$ geodesic is uniquely determined by an edge in $E(V_1, V_2)$. Hence
\[
n_G(x, y) = |E(V_1, V_2)|.
\]
Since $G$ is simple, every vertex in $V_1$ has exactly one neighbor in $V_0=\{x\}$, and thus has at most $\Delta-1$
neighbors in $V_2$. Moreover, $|V_1|\leq \Delta$. It follows that
\[
n_G(x, y) = |E(V_1, V_2)|
\leq |V_1|(\Delta-1)  \le \Delta (\Delta - 1).
\]
Equality holds if and only if
\[
|V_1|=|V_2|=\Delta
\]
and $H[V_1, V_2]$ is obtained from $K_{\Delta,\Delta}$ by deleting a perfect matching.

We henceforth assume that $t\geq4$. We follow the same probabilistic argument as in the previous section, with the modifications arising from the assumption that $G$ is simple. Since $V_0= \{x \}$ and $V_t=\{y\}$, the simplicity of $G$ implies
\[
d^-(v)=1
\quad\text{and}\quad
d^+(v)\leq \Delta-1
\quad\text{for every }v\in V_1,
\]
and similarly,
\[
d^+(v)=1
\quad\text{and}\quad
d^-(v)\leq \Delta-1
\quad\text{for every }v\in V_{t-1}.
\]

Let $\mathcal P$ be the set of all $x$--$y$ geodesics. As in the multigraph case, define two probability distributions
$\mu_{x\to y}$ and $\mu_{y\to x}$ on $\mathcal P$ by the random
walks from $x$ to $y$ and from $y$ to $x$, respectively. Then for each
\[
P = (x = v_0, v_1, \ldots, v_{t-1}, v_t = y)\in \mathcal P,
\]
we have
\[
\mu_{x \to y}(P) = \frac{1}{d^+(x) d^+(v_1) \prod_{i=2}^{t-2} d^+(v_i)} \ge \frac{1}{\Delta (\Delta - 1)} \cdot \frac{1}{ \prod_{i=2}^{t-2} d^+(v_i)}
\]
and
\[
\mu_{y \to x}(P) = \frac{1}{d^-(y) d^-(v_{t-1}) \prod_{i=2}^{t-2} d^-(v_i)} \ge \frac{1}{\Delta (\Delta - 1)} \cdot \frac{1}{ \prod_{i=2}^{t-2} d^-(v_i)}.
\]
Now, set
\[
        a_i:=d^+(v_i),
        \quad
        b_i:=d^-(v_i), \quad 2 \le i \le t-2,
\]
and
\[
        A'(P):=\prod_{i=2}^{t-2}a_i,
        \qquad
        B'(P):=\prod_{i=2}^{t-2}b_i.
\]
Then
\[
a_i + b_i = d^+(v_i) + d^-(v_i) \le \Delta\quad\text{for every }2 \le i \le t-2.
\]
As before, we consider the probability distribution $\mu = \frac12(\mu_{x \to y} + \mu_{y \to x})$. Applying Lemma~\ref{lem:reciprocal-product} with $r  = t -3$, we obtain, for every $P \in \mathcal P$,
\begin{align*}
\mu(P) \ge \frac{1}{2 \Delta (\Delta - 1)} \left( \frac{1}{A'(P)} + \frac{1}{B'(P)}\right) \ge \begin{cases}
\Delta^{-1} (\Delta - 1)^{-1} q^{-(t-3)/2},
& \text{if $t$ is odd},\\[2mm]
2^{-1} (\Delta - 1)^{-1} q^{-(t-2)/2},
& \text{if $t$ is even}.
\end{cases}
\end{align*}
Therefore, by the probabilistic counting bound \eqref{eq:pro-counting},
\[
n_G(x, y) = |\mathcal P|  \le \begin{cases}
\Delta (\Delta - 1) q^{(t-3)/2},
& \text{if $t$ is odd},\\[2mm]
2 (\Delta - 1) q^{(t-2)/2},
& \text{if $t$ is even}.
\end{cases}
\]

\subsection{Equality cases and tight examples}

We next examine the structure of $H$ when equality holds and give tight examples. By the preceding argument, equality implies
\[
|V_1| = |V_{t-1}| = \Delta,
\]
and the edge sets $E(V_0, V_1)$ and $E(V_{t-1}, V_t)$ are uniquely determined. Hence it remains to determine the bipartite graphs $H[V_i, V_{i+1}]$ for $1\le i \le t-2$. We distinguish two cases.

\medskip

\noindent\textbf{Case 1: $\Delta=2k$.} By the equality conditions in Lemma~\ref{lem:reciprocal-product}, every $v \in V_i$ with $2\leq i\leq t-2$ satisfies
\[
d^-(v) = d^+(v) = k.
\]
Moreover, equality requires $d^+(v) = 2k-1$ for every $v \in V_1$. Since $|V_1|=2k$, counting $|E(V_1, V_2)|$ gives
\[
(2k - 1) |V_1| = k |V_2|,
\]
and hence $|V_2| = 2(2k-1)$. Similarly, $|V_{t-2}| = 2(2k-1)$. For $2\le i \le t-3$, counting $|E(V_{i}, V_{i+1})|$ from the two sides gives
\[
|V_i| = |V_{i+1}|.
\]
Consequently,
\[
|V_i| = 2(2k-1)
\quad\text{for every }2\le i \le t-2.
\]

The equality conditions further determine the degrees in the bipartite
graphs between consecutive layers. In particular,
$H[V_1,V_2]$ is $(2k-1,k)$-biregular, with degree $2k-1$ on $V_1$ and degree $k$ on $V_2$. Similarly,
$H[V_{t-2},V_{t-1}]$ is $(k,2k-1)$-biregular. Such simple biregular bipartite graphs exist, for example, by standard cyclic constructions. For every $2\leq i\leq t-3$, the bipartite graph $H[V_i,V_{i+1}]$ is $k$-regular. Conversely, any choice of bipartite graphs between consecutive layers satisfying the above degree conditions yields a tight example.

\medskip

\noindent\textbf{Case 2: $\Delta=2k+1$.} For $v\in V_i$ with $2\leq i\leq t-2$, we say that $v$ is of
type $A$ if
\[
(d^-(v),d^+(v))=(k,k+1),
\]
and of type $B$ if
\[
(d^-(v),d^+(v))=(k+1,k).
\]
By the equality conditions in Lemma~\ref{lem:reciprocal-product},
every vertex in $V_i$, $2\leq i\leq t-2$, is of one of these two types; let $p_i$ and $q_i$ denote the respective numbers of vertices of types $A$ and $B$.

Since $|V_1|=\Delta=2k+1$ and equality requires
$d^+(v)=2k$ for every $v\in V_1$, counting
$|E(V_1,V_2)|$ gives
\[
kp_2+(k+1)q_2=2k(2k+1).
\]
Similarly, counting $|E(V_{t-2},V_{t-1})|$ gives
\[
(k+1)p_{t-2}+kq_{t-2}=2k(2k+1).
\]
For $2\leq i\leq t-3$, counting $|E(V_i,V_{i+1})|$
from the two sides gives
\[
(k+1)p_i+kq_i
=
kp_{i+1}+(k+1)q_{i+1}.
\]
Thus the sequence $(p_i,q_i)$ must satisfy
these recurrence and boundary conditions.

Unlike the case where $\Delta$ is even, the bipartite graphs $H[V_i,V_{i+1}]$ need not be biregular. Once the sequence
$(p_i,q_i)$, $2\leq i\leq t-2$, is fixed, the bipartite graphs
between consecutive layers must realize the following prescribed
degrees. In $H[V_1,V_2]$, every vertex in $V_1$ has degree $2k$, while
$p_2$ vertices in $V_2$ have degree $k$ and $q_2$ vertices have
degree $k+1$. Similarly, in $H[V_{t-2},V_{t-1}]$, every vertex in
$V_{t-1}$ has degree $2k$, while $q_{t-2}$ vertices in $V_{t-2}$
have degree $k$ and $p_{t-2}$ vertices have degree $k+1$. For $2\leq i\leq t-3$, in $H[V_i,V_{i+1}]$, there are $q_i$
vertices of degree $k$ and $p_i$ vertices of degree $k+1$ in
$V_i$, whereas there are $p_{i+1}$ vertices of degree $k$ and
$q_{i+1}$ vertices of degree $k+1$ in $V_{i+1}$. In addition, along every $x$--$y$ geodesic, the numbers of
vertices of types $A$ and $B$ lying in $V_2,\ldots,V_{t-2}$
must be equal when $t$ is odd and differ by exactly $1$ when
$t$ is even.

We now construct tight examples. Since tracking the numbers of vertices of types $A$ and $B$ along every $x$--$y$ geodesic can be cumbersome, we arrange the two types to alternate along each geodesic, so that the required balance is automatic.

We choose $(p_i,q_i) = (2k, 2k)$ for every $2\le i \le t-2$. This sequence satisfies the boundary and recurrence conditions derived above. Thus $|V_i|=4k$ for every $2\leq i\leq t-2$.

Then we split the construction into two internally disjoint channels. Write
\[
V_1=V_1^{(1)} \bigsqcup V_1^{(2)},
\qquad
|V_1^{(1)}|=k,\qquad |V_1^{(2)}|=k+1,
\]
For each $2\leq i\leq t-2$, partition
\[
V_i=V_i^{(1)}\bigsqcup V_i^{(2)},
\qquad
|V_i^{(1)}|=|V_i^{(2)}|=2k.
\]
In the first channel, the types in
$V_2^{(1)},\ldots,V_{t-2}^{(1)}$ alternate starting with $A$,
whereas in the second channel they alternate starting with $B$.

There are no edges between the two channels. At the left end, take
\[
H[V_1^{(1)},V_2^{(1)}]=K_{k,2k},
\qquad
H[V_1^{(2)},V_2^{(2)}]=K_{k+1,2k}.
\]
For $2\leq i\leq t-3$ and $j\in\{1,2\}$, choose
$H[V_i^{(j)},V_{i+1}^{(j)}]$ to be $(k+1)$-regular if
$V_i^{(j)}$ is of type $A$, and $k$-regular if it is of type $B$.

It remains to specify the right end. For $j\in\{1,2\}$, let
\[
|V_{t-1}^{(j)}|=
\begin{cases}
k+1, & \text{if $V_{t-2}^{(j)}$ is of type $A$,}\\
k,   & \text{if $V_{t-2}^{(j)}$ is of type $B$,}
\end{cases}
\]
and take
\[
H[V_{t-2}^{(j)},V_{t-1}^{(j)}]
=
K_{\,2k,\,|V_{t-1}^{(j)}|}.
\]

Along every $x$--$y$ geodesic, the types therefore alternate.
If $t$ is odd, the number $t-3$ of layers
$V_2,\ldots,V_{t-2}$ is even, so the two types occur equally often.
If $t$ is even, $t-3$ is odd, so one of the two types occurs exactly
once more than the other. Hence the equality conditions in
Lemma~\ref{lem:reciprocal-product} are satisfied in both cases.
Therefore, this construction attains the bound.

\medskip

We have thus proved the stated bounds, described the structure of $H$ when equality holds, and provided tight examples. This completes the proof.

\section*{Acknowledgements}
J.Y. acknowledges partial support by the National Natural Science Foundation of China grant 12371343 and 12525110 (PI: Hehui Wu).
J.-X.Z. acknowledges support from the Natural Science Foundation of Shanghai (25ZR1402414) and the NSFC (12501185). The authors used ChatGPT (OpenAI) during the development of this work for exploratory discussions of proof strategies and for language assistance. All mathematical arguments were independently verified and formalized by the authors.

\section*{Conflict of Interest}

The authors declare no conflicts of interest.

\section*{Data Availability Statement}

Data sharing is not applicable to this article, as no datasets were generated or analyzed during the current study.


\begin{thebibliography}{99}

\bibitem[BR14]{BabuRadhakrishnan2014}
S.~A. Babu and J. Radhakrishnan,
\emph{An entropy-based proof for the Moore bound for irregular graphs},
in \emph{Perspectives in Computational Complexity},
Progr. Comput. Sci. Appl. Logic, vol.~26,
Birkh\"auser/Springer, Cham, 2014, pp.~173--181.

\bibitem[BHOW11]{BHOW11}
I. Benjamini, C. Hoppen, E. Ofek, P. Pra{\l}at, and N. Wormald,
\emph{Geodesics and almost geodesic cycles in random regular graphs},
J. Graph Theory \textbf{66} (2011), no.~2, 115--136.
\href{https://doi.org/10.1002/jgt.20496}{doi:10.1002/jgt.20496}.

\bibitem[BT23]{ben2023}
I. Benjamini and E. Tzalik,
\emph{On the number of shortest paths in graphs},
preprint, 2023,
\url{https://arxiv.org/abs/2311.10014}.

\bibitem[KSSZ26]{KSSZ26}
M. Knor, J. Sedlar, R. \v{S}krekovski, and X.-D. Zhang,
\emph{Counting geodesic paths in graphs},
preprint, 2026,
\url{https://arxiv.org/abs/2604.04907}.

\bibitem[Lub66]{lubell1966}
D. Lubell,
\emph{A short proof of Sperner's lemma},
J. Combinatorial Theory \textbf{1} (1966), 299.

\end{thebibliography}
\end{document}